\documentclass[11pt]{amsart}

\title[Homology from Trisection Diagrams]{Calculating the homology and intersection form of a 4-manifold from a trisection diagram}

\author{Peter Feller}
\author{Michael Klug}
\author{Trent Schirmer}
\author{Drew Zemke}

\usepackage{amsmath}
\usepackage{amssymb}
\usepackage{amsthm}
\usepackage{MnSymbol}
\usepackage[percent]{overpic}
\usepackage{tikz-cd}

\usepackage[margin=3.5cm]{geometry}
\usepackage{parskip}

\newtheorem{theorem}{Theorem}[section]
\newtheorem{proposition}[theorem]{Proposition}
\newtheorem{lemma}[theorem]{Lemma}
\newtheorem{corollary}[theorem]{Corollary}

\theoremstyle{definition}
\newtheorem{definition}[theorem]{Definition}

\newtheorem{rmk}[theorem]{Remark}

\newtheorem*{example*}{Example}

\numberwithin{equation}{section}

\newcommand{\R}{\mathbb{R}}

\newcommand{\Z}{\mathbb{Z}}

\newcommand{\dmat}[4]
{\left(
    \begin{array}{cc}
        #1 & #2 \\
        #3 & #4
    \end{array}
\right)}

\begin{document}

\begin{abstract}
Given a diagram for a trisection of a 4-manifold $X$, we describe the
homology and the intersection form of $X$ in terms of the three subgroups of $H_1(\Sigma;\mathbb{Z})$ generated by the three sets of curves and the
intersection pairing on the diagram surface $\Sigma$. This includes explicit
formulas for the second and third homology groups of $X$ as well an algorithm to compute the intersection form.  Moreover, we show that all $(g;k,0,0)$-trisections admit ``algebraically trivial'' diagrams.
\end{abstract}

\maketitle

\section{Introduction}

In this note\footnote{This text has been submitted to appear as part of the Proceedings of the National Academy of Sciences collection on trisections of 4-manifolds.}, we describe how to compute the homology and intersection form of a $4$-manifold using the algebraic intersection numbers of the curves appearing in any trisection diagram for the $4$-manifold.  Moreover, we show that a large class of trisections admit ``algebraically trivial'' diagrams.

\subsection{Homology and the intersection form} The \emph{homology} of an $n$-dimensional manifold $X$ is a collection of $n+1$ abelian groups $H_0(X),H_1(X),\ldots , H_n(X)$.  One may think of the elements of $H_i(X)$ as equivalence classes of oriented closed $i$-dimensional submanifolds immersed in $X$--where two such $i$-manifolds $Y_1$ and $Y_2$ are considered equivalent if there is an $(i+1)$-dimensional manifold immersed in $X$ which is bounded by $Y_1$ and $Y_2$.  If $\Sigma$ is a closed surface, for example, then the elements of $H_1(\Sigma)$ can be thought of as immersed oriented curves in $\Sigma$, where an orientation in this case amounts to assigning a direction to a curve.  Two sets of curves cancel one another out if they bound an immersed subsurface with orientations which match the boundary-orientation of the surface.  Addition of homology classes corresponds to taking the union of curves.

In the case of a $4$-manifold $X$, the elements of $H_2(X)$ correspond to embedded surfaces inside of $X$.  Surfaces embedded in a $4$-manifold generically intersect one another in points. 
(To see why, consider the intersection of the $x$-$y$ plane with the $z$-$w$ plane in $\R^4;$ in fact, this is an appropriate local model for a generic intersection of surfaces in a $4$-manifold.) 
The \emph{intersection form} of a $4$-manifold is a symmetric, bilinear product on $H_2(X)$ which, given a pair of surfaces $\Sigma_1$ and $\Sigma_2$ as input, outputs the ``algebraic count'' of the intersections between them.

\subsection{Motivation}

Whenever a new description of a topological object arises, it is natural to ask how this can be used to describe the fundamental algebraic invariants of the space, such as its homology and intersection form.  The homology and intersection forms of $4$-manifolds reveal an interesting dichotomy between the smooth and the topological category of 4-manifolds.  On the one hand, results such as Rokhlin's theorem~\cite{Rochlin} and Donaldson's diagonalization theorem~\cite{Donaldson} impose strong restrictions on the intersection form of \textit{smooth} closed oriented $4$-manifolds; but on the other hand \emph{all} symmetric unimodular bilinear forms occur as the intersection form of some \textit{topological} closed oriented $4$-manifolds, by Freedman's work~\cite{Freedman}. Since every smooth $4$-manifold admits a trisection~\cite{GayKirby}, it is worthwhile to be able to compute the homology and the intersection form of a $4$-manifold directly from a trisection diagram.  One hope is that Rohklin and Donaldson's strong restrictions on intersection forms will become more apparent from this new perspective.  More generally, we expect the calculation of these elementary yet fundamental invariants to play an important role in the continued development of the theory of trisections.

\subsection{Results}

Let $(\Sigma;V_1,V_2,V_3)$ denote a $(g;k_1,k_2,k_3)$-trisection of a smooth closed oriented $4$-manifold $X$, and let $W_\alpha=V_1\cap V_2$, $W_\beta=V_1\cap V_3,$ and $W_\gamma=V_2\cap V_3$.  For
$\epsilon \in \{\alpha,\beta,\gamma\}$ let $L_\epsilon=\ker(\iota_\epsilon)$, where $\iota_\epsilon: H_1(\Sigma)\rightarrow H_1(W_\epsilon)$ is the map induced by inclusion.  The groups $L_\epsilon$ are maximal Lagrangian subgroups of $H_1(\Sigma)$ generated by any choice of oriented defining curves for $W_\epsilon$.

We show that the following chain complex gives rise to the homology of $X$.  (The groups depicted are in dimensions four through zero, and the complex is understood to have trivial groups and maps in other dimensions.)
\begin{equation}
	\label{eqintro:4dHomology}
    \Z \xrightarrow{0} L_\alpha\cap L_\gamma
    	\xrightarrow{\partial_3}\frac{L_\gamma}{L_\beta\cap L_\gamma}
        \xrightarrow{\partial_2}\textnormal{Hom}(L_\alpha\cap L_\beta,\Z)
        \xrightarrow{0}\Z,
\end{equation}

The map $\partial_3$ is induced by the inclusion $L_\alpha\cap L_\gamma\hookrightarrow L_\gamma$ and
$\partial_2$ is given by $[x]\mapsto \langle \cdot, x\rangle_\Sigma$ ($\langle\cdot,\cdot\rangle_\Sigma$ denotes the intersection form on $H_1(\Sigma,\Z)$).

We can also describe the homology of $X$ in a way that is symmetric in the three Lagrangian subspaces
$L_\epsilon$.  As shown in Corollary \ref{cor:H2Formula}, the second homology group of $X$ can be identified with the group
\[
	\frac{\{a+b+c=0\}}{\{ c=0\}+\{b=0\}+\{a=0\}},
\]
where $\{a+b+c=0\}$ denotes the subgroup of $L_\alpha\times L_\beta\times L_\gamma$ of elements
$(a,b,c)$ such that $a+b+c=0$, and the summands in the denominators are likewise understood to be the subgroups of
$\{a+b+c=0\}$ satisfying the given linear equations.  With respect to this symmetric presentation of $H_2(X)$, the intersection form is given by the following equation.
   \begin{equation}\label{eqintro:pairing}
  	\left\langle(a,b,c),(a',b',c')\right\rangle
  		=\langle a,b'\rangle_\Sigma
  \end{equation}
In Section \ref{sec:applications} we describe various methods of computing the intersection form directly from a trisection diagram.  In particular, Propositions \ref{prop:GeneralMatrix} and \ref{prop:LovelyProduct} give explicit formulas in terms of the three algebraic intersection matrices arising from the defining curves of the trisection diagram.  The following example demonstrates one of our methods.

\begin{example*}
Let us compute the intersection form of the $(2;0,0,0)$-trisection
    of $S^2 \times S^2$ with diagram shown in Figure \ref{fig:s2s2diagram}.
    Using the $\alpha_i$, $\beta_i$, $\gamma_i$ as labeled in the figure, we obtain
    the following intersection matrices.
    \begin{align*}
        &_\alpha Q_\beta = \dmat{0}{-1}{1}{0},
            \:
        _\gamma Q_\beta = \dmat{0}{-1}{-1}{0},
            \:
        _\alpha Q_\gamma = \dmat{0}{1}{-1}{0}
    \end{align*}
    Using Proposition \ref{prop:LovelyProduct}, we find that
    the intersection form of $S^2 \times S^2$ has matrix
    \[
        \Phi_{S^2 \times S^2} = {_\gamma}Q_\beta({_\alpha}Q_\beta)^{-1}{_\alpha}Q_\gamma
            = \dmat{0}{1}{1}{0},
    \]
    as expected.
\end{example*}

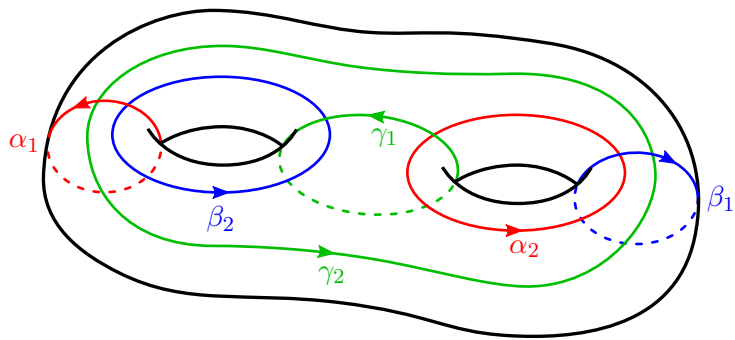
\begin{figure}
	\begin{center}
    \begin{overpic}[width=0.6\textwidth]{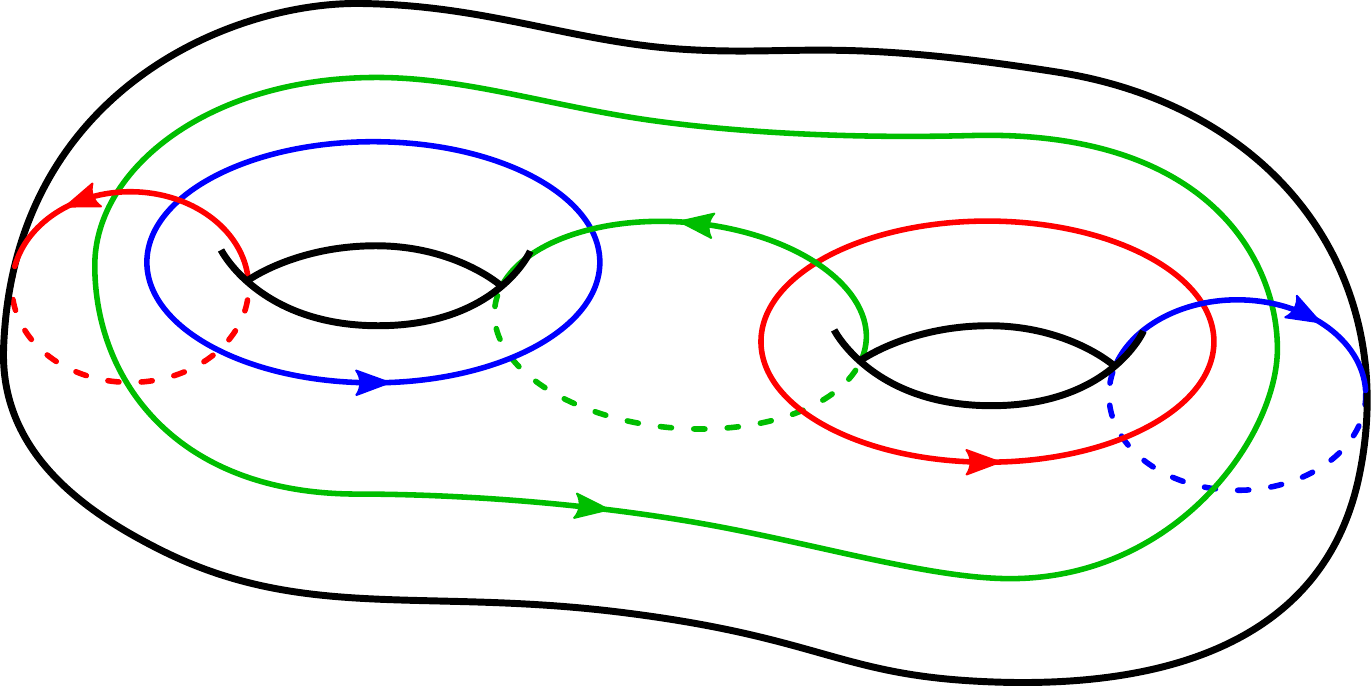}
    	\definecolor{mygreen}{rgb}{0, 0.745, 0}
        \put (-5,29) {\color{red}$\alpha_1$}
        \put (71,13) {\color{red}$\alpha_2$}
        \put (101,20) {\color{blue}$\beta_1$}
        \put (25,17) {\color{blue}$\beta_2$}
        \put (50,30) {\color{mygreen}$\gamma_1$}
        \put (42,9) {\color{mygreen}$\gamma_2$}
    \end{overpic}
    \end{center}
    \caption{A $(2;0,0,0)$-trisection diagram for $S^2 \times S^2$.}
    \label{fig:s2s2diagram}
\end{figure}

%

The strategy we use to compute the homology of $X$ is to establish that the complex in~\eqref{eqintro:4dHomology} is isomorphic to a chain complex arising from a handle decomposition of the trisected $4$-manifold $X$. However, as an amusing exercise, one may use the fact that any two trisection diagrams for $X$ are related by handle slides and stabilizations \cite{GayKirby} to show that all chain complexes of the form of~\eqref{eqintro:4dHomology} coming from trisections of $X$ are chain homotopy equivalent.
Similarly, one many directly check that for any two trisections of $X$ the pairings described in~\ref{eqintro:pairing} are isomorphic.

Finally, Theorem \ref{thm:AlgebraicTriviality} asserts that all $(g;k,0,0)$-trisections admit ``algebraically trivial'' diagrams.  One consequence of this is that the triple of algebraic intersection matrices arising from a trisection diagram do not contain any additional information beyond the homology and intersection form (thus giving a negative answer to a speculative question appearing in \cite{GayKirby}, at least for this class of trisections).

One may speculate on whether Theorem \ref{thm:AlgebraicTriviality} has further applications.
There is a large difference between algebraically trivial diagrams and geometrically trivial diagrams--a gap which is reminiscient of the failure of the h-cobordism theorem for cobordisms between $4$-manifolds.  On the other hand, if new restrictions on intersection forms are to be found using trisection diagrams (or if old, previously known restrictions are to be recovered more easily), then it seems likely that Theorem \ref{thm:AlgebraicTriviality} will play a role.  Additionally, the result suggests that the relationship between the Torelli group and the Georitz group of a surface might play an important role in understanding exotic $4$-dimensional phenomena.

\subsection{Acknowledgements}
The authors would like to thank AIM and the organizers of its March 2016 conference on trisections, without which this collaboration would not have occurred.  The authors would also like to thank Rob Kirby, Jeff Meier, and Matthias Nagel for helpful comments and conversations.


\section{Three-Dimensional Lemmas}

Our four-dimensional calculations will depend on our ability to compute intersections and linking numbers between various attaching circles and belt spheres embedded in $S^3$ or $\#^k(S^1\times S^2)$.  The lemmas of this section describe how to carry out these calculations inside the first homology of a Heegaard
splitting surface in one of these 3-manifolds.

For a closed 3-manifold $M$ we denote by $\langle \cdot,\cdot \rangle_M$ the intersection pairing
	on $H_2(M) \times H_1(M)$, and for a closed surface $\Sigma$ we denote by
    $\langle \cdot,\cdot \rangle_\Sigma$ the intersection pairing on $H_1(\Sigma) \times H_1(\Sigma)$.

\begin{lemma}\label{lem:Lagrangian}
Suppose $W$ is a three-dimensional handlebody with boundary surface $\Sigma$, and let $\iota_*:H_1(\Sigma)\rightarrow H_1(W)$ be the map induced by inclusion. Then:

\begin{enumerate}
	\item If $\{ \alpha_1,\ldots , \alpha_g\}$ is any defining set of curves for $W$ in $\Sigma$, then $\{ [\alpha_1],\ldots , [\alpha_g]\}$ forms a basis of $\ker(\iota_*),$ and

	\item For all $x,y\in \ker(\iota_*)$, $\langle x,y\rangle_\Sigma=0$ and $\ker(\iota_*)$ is maximal with respect to this property.  In other words, $\ker(\iota_*)$ is a maximal Lagrangian subspace of $H_1(\Sigma)$.
\end{enumerate}
\end{lemma}

\begin{proof}
Since $H_1(\Sigma)\cong \Z^{2g}$, $H_1(W)\cong \Z^g$, and $\iota_*$ is a surjection, $\ker(\iota_*)$ is a free abelian subgroup of rank $g$.  Since $\{\alpha_1,\ldots , \alpha_g\}$ is a defining set of curves, the set $\{[\alpha_1],\ldots , [\alpha_g]\}$ is linearly independent in $H_1(\Sigma)$ and the subgroup it generates lies inside of $\ker(\iota_*)$. This proves Claim (1).

If $x,y\in \ker(\iota_*)$, then by Claim (1) $x$ and $y$ can be represented by disjoint collections of curves isotopic to the elements of $\{\alpha_1,\ldots ,\alpha_g\},$ and so it follows that $\langle x,y\rangle_\Sigma=0$.
Moreover, if $z$ is any element of $H_1(\Sigma)$ satisfying $\langle x,z\rangle_\Sigma=0$ for all $x\in \ker(\iota_*)$, then any representative $\zeta$ of $z$ can be band-summed along subarcs of the $\alpha$-curves to a new representative $\zeta'$ of $z$ which is disjoint from the $\alpha$-curves. This implies that $z\in \ker(\iota_*)$, which establishes the claim that $\ker(\iota_*)$ is maximal.  Alternatively, maximality follows from the observation that the $\alpha$-curves form half of a symplectic basis of $H_1(\Sigma)$.
\end{proof}

For the remainder of this section, let $(\Sigma; W_\alpha,W_\beta)$ be a Heegaard splitting of a $3$-manifold $M,$ let $\iota_M:H_1(\Sigma)\rightarrow H_1(M)$ and $\iota_\epsilon:H_1(\Sigma)\rightarrow H_1(W_\epsilon)$ denote the maps induced by inclusion (for $\epsilon \in \{\alpha, \beta\}$),
and set $L_\epsilon=\ker(\iota_\epsilon).$

For the following lemma, it is necessary to have an understanding of the Mayer-Vietoris map $\partial_*:H_2(M)\rightarrow H_1(\Sigma)$. Suppose $Q\subset M$ is an oriented embedded surface which represents $x\in H_2(M)$ and is transverse to $\Sigma$.  Then  $\partial_*x$ is represented by the family of curves of $Q\cap \Sigma$ with the orientations they inherit as the boundary of the oriented surface $Q\cap W_\alpha$ using the standard outward normal first convention.


\begin{lemma}\label{lem:3dHomology}
Let $\partial:H_2(M)\rightarrow H_1(\Sigma)$ be the boundary map coming from the Mayer-Vietoris sequence for the triple $(\Sigma;W_\alpha,W_\beta)$.  Then:
\begin{enumerate}
	\item The map $\partial$ is an isomorphism between $H_2(M)$ and $L_\alpha\cap L_\beta$;

	\item For any $z\in H_2(M)$ and $y\in H_1(\Sigma)$,
    	$\langle\iota_My,z\rangle_M = \langle\partial z, y\rangle_\Sigma$;

	\item If $H_1(M)$ is torsion free and $x \in H_1(\Sigma)$, then $x\in\ker(\iota_M)$ if and only if $\langle y,x\rangle_\Sigma=0$ for all $y\in L_\alpha\cap  L_\beta.$
\end{enumerate}
\end{lemma}

\begin{proof}
In the following portion of the Mayer-Vietoris sequence
\[ 
\begin{tikzcd}
	H_2(W_\alpha)\oplus H_2(W_\beta) \ar[r]
    & H_2(M) 
        \ar[r,"\partial"]
    & H_1(\Sigma) \ar[r, "\iota_\alpha\oplus\iota_\beta"] & H_1(W_\alpha)\oplus H_1(W_\beta)
\end{tikzcd}
\]
we have $H_2(W_\alpha)\oplus H_2(W_\beta)= 0$ and $\ker(\iota_\alpha\oplus\iota_\beta)= L_\alpha\cap L_\beta$, so by exactness $\partial$ is an injective map onto $L_\alpha\cap L_\beta.$

For the second claim, choose a representative $Q$ of $x$ which is transverse to $\Sigma$, and a representative $\gamma$ of $y$ which is transverse to $Q\cap \Sigma$.  Then $\gamma$ also represents $\iota_M y$, $Q\cap \Sigma$ represents $\partial_*x$, and since all intersections between $\gamma$ and $Q$ lie in $Q\cap \Sigma$ it follows that $\langle x,\iota_My\rangle_M$ agrees with $\langle\partial x,y\rangle_\Sigma$ up to sign.  The question of whether their signs agree is then a technical issue of orientation conventions.  Our convention leads to the requirement that $\partial_*x$ must always appear on the left hand side of the product $\langle\cdot,\cdot\rangle_\Sigma$.

For Claim (3), note Claims (1) and (2) imply that $\langle y,x\rangle_\Sigma =0$ for all $y\in L_\alpha\cap L_\beta$ if and only if $\langle \partial z,x\rangle_\Sigma=\langle z,\iota_Mx\rangle_M =0$ for all $z\in H_2(M)$.  This is the same as saying that $\iota_Mx$ lies in the kernel of the Poincar\'{e} duality map $PD:H_1(M)\rightarrow H^2(M)$, which is defined by the formula $PD(a)=\langle \cdot , a\rangle_M$ when $H_1(M)$ is torsion free.  Since $PD$ is an isomorphism, it has a trivial kernel, hence $\iota_Mx=0.$
\end{proof}

The intersection form of a $4$-manifold which admits a handle decomposition with no $1$-handles or $3$-handles is the same as the linking matrix of the framed link formed by the attaching circles of the $2$-handles \cite{GompfStipsicz}.  In the general case, in which a handle decomposition may have $1$- and $3$-handles, the intersection form can still be computed using a linking matrix, but it requires a slightly more general definition of a linking number that applies to null-homologous pairs of links in arbitrary $3$-manifolds (instead of just pairs of knots in $S^3$).

\begin{definition}\label{def:LinkingNumber}
Let $J$ and $K$ be oriented, null-homologous links in a $3$-manifold $M$, and let $F$ be a Seifert surface of $J$.  Then $\textnormal{lk}(J,K)=\langle F,K\rangle_M.$
\end{definition}

Here, a \emph{Seifert surface} of an oriented link $J$ is an embedded, oriented surface in $M$ whose boundary is $J$ (with matching induced orientation).  In the definition above, it is necessary that $K$ be null-homologous in $M$ as well, otherwise the linking number is not defined.  The usual properties of linking number (such as symmetry) all hold.

\begin{lemma}\label{lem:LinkingNumber}
Suppose $J$ is an embedded oriented union of curves in $\Sigma$ which forms a null-homologous link in $M$.
\begin{enumerate}
	\item There exists $j\in L_\alpha$ such that $\langle j,x\rangle_\Sigma = \langle [J],x\rangle_\Sigma$ for all $x\in  L_\beta$.
    \item For $j \in L_\alpha$ satisfying Claim (1), and for any other embedded oriented union of curves $K\subset \Sigma\setminus J$ which is null-homologous in $M$, $\textnormal{lk}(J,K)=\langle j,[K]\rangle_\Sigma$.
\end{enumerate}
\end{lemma}

\begin{proof}
Let $J'$ denote the result of pushing $J$ slightly into the interior of $W_\beta,$ so that $J'$ and $J$ cobound an embedded union of annuli in $W_\beta.$  Let $F\subset M$ be a Seifert surface for $J'$ that is transverse to $\Sigma$, and set $j=[F\cap \Sigma]$, oriented as the boundary of $F\cap W_\alpha$ as in Lemma \ref{lem:3dHomology}.  Then $j\in L_\alpha$ since it bounds the $2$-chain $[F\cap W_\alpha]$, and since $\iota_\beta([J])=[J']=\iota_\beta(j),$ we have $[J]-j\in  L_\beta.$  Now by Lemma \ref{lem:Lagrangian}, it follows that $\langle [J]-j,x\rangle_\Sigma=\langle [J],x\rangle_\Sigma-\langle j, x\rangle_\Sigma =0$ for all $x\in  L_\beta$, which proves Claim (1).

To prove Claim (2), first notice that for any null-homologous $K$, the $j$ from the previous paragraph
	satisfies
	\[
    	\textnormal{lk}(J,K)
        = \langle [F],[K]\rangle_{(M,J)}
        =\langle [F\cap \Sigma],[K]\rangle_\Sigma
        = \langle j,[K]\rangle_\Sigma,
	\]
where $\langle\cdot,\cdot\rangle_{(M,J)}$ is the intersection pairing on the relative homology
$H_1(M,J) \times H_2(M,J)$.
If $j'\in L_\alpha$ is any other element which satisfies $\langle j',x\rangle_\Sigma=\langle [J],x\rangle_\Sigma$ for all $x\in  L_\beta$, then we have $\langle j-j',x\rangle_\Sigma=0$ for all $x\in L_\beta$.  By Lemma \ref{lem:Lagrangian} it follows that $j-j'\in L_\beta$, and since $j-j'\in L_\alpha$ we have $j-j'\in L_\alpha\cap L_\beta.$  By Claim (3) of Lemma \ref{lem:3dHomology},
	$\langle j-j',[K]\rangle_\Sigma=0$ for any embedded union of curves $K\subset \Sigma \setminus J$ which is null-homologous in $M$. This proves Claim (2).
\end{proof}


\section{The homology and intersection form of a trisected 4-manifold}

In this section, let $(\Sigma;V_1,V_2,V_3)$ denote a $(g;k_1,k_2,k_3)$-trisection of the $4$-manifold $X$, let $W_\alpha=V_1\cap V_2$, $W_\beta=V_1\cap V_3,$ and $W_\gamma=V_2\cap V_3$, and let $L_\epsilon=\ker(\iota_\epsilon)$, where $\iota_\epsilon: H_1(\Sigma)\rightarrow H_1(W_\epsilon)$ is the map induced by inclusion, $\epsilon=\alpha,\beta,\gamma$.  As discussed in the previous section, the groups $L_\epsilon$ are maximal Lagrangian subgroups of $H_1(\Sigma)$ generated by any choice of oriented defining curves for $W_\epsilon$.  We now show how the homology and intersection form of $X$ can be calculated in terms of the intersection form on $H_1(\Sigma)$ and the subgroups $L_\epsilon.$

 \begin{theorem}\label{thm:4dHomology}

 The homology of $X$ can be obtained from the following chain complex:
\[ 
\begin{tikzcd}
	0 \rightarrow
    	\Z  \ar[r, "0"]
    & (L_\alpha \cap L_\gamma) \oplus (L_\beta \cap L_\gamma) \ar[r, "\partial_3"]
    & L_\gamma \ar[r, "\partial_2"]
    & \textnormal{Hom}(L_\alpha\cap L_\beta,\Z) \xrightarrow{0}
    \Z \rightarrow 0
\end{tikzcd}
\]
where $\partial_3(x,y)=x+y$ and $\partial_2(x)=\langle \cdot,x\rangle_\Sigma$.
In fact, this chain complex is isomorphic to the cellular chain complex of $X$ obtained from a
	certain handle decomposition of $X$.
\end{theorem}

\begin{rmk}\label{rmk:4dHomology}
Alternatively, the homology of $X$ can be obtained from the following chain complex, which occurs in the introduction:
\begin{align}
	\label{eq:4dHomology}
    0
   	\rightarrow \Z
        \xrightarrow{0} L_\alpha\cap L_\gamma
        \xrightarrow{\partial_3} \frac{L_\gamma}{L_\beta\cap L_\gamma}
        \xrightarrow{\partial_2} \textnormal{Hom}(L_\alpha\cap L_\beta,\Z)
        \xrightarrow{0} \Z
        \rightarrow 0,
\end{align}
where $\partial_3$ is induced by inclusion and $\partial_2$ is the factorization through
	$\frac{L_\gamma}{L_\beta\cap L_\gamma}$ of the corresponding differential
    $\partial_2$ in Theorem~\ref{thm:4dHomology}.
Indeed, noting that $\frac{L_\gamma}{L_\beta\cap L_\gamma}$ can be (non-canonically) identified with a direct summand in
$L_\gamma$ such that $L_\gamma=L_\beta\cap L_\gamma\oplus\frac{L_\gamma}{L_\beta\cap L_\gamma}$, one can show that the two chain complexes are chain homotopy equivalent.

\end{rmk}
Below, Theorem~\ref{thm:4dHomology} is established from a  handle decomposition of $X$ adapted to pieces of its trisection which is built directly from the definition of a trisection. Alternatively, one can check that if a trisection comes from a Heegaard-Kirby diagram for $X$ (as described in \cite[Proposition~4.2]{MeierSchirmerZupan}), then the $CW$-homology chain complex obtained from the resulting handle decomposition is isomorphic to the complex in~\eqref{eq:4dHomology}.

\begin{proof}[Proof of Theorem~\ref{thm:4dHomology}]
Given a handle decomposition of $X$, let $X^i\subset X$ denote the union of all handles of index less than or equal to $i$.  We build a handle decomposition of $X$ as follows: start with a single $0$-handle and $k_1$ $1$-handles in $V_1$, so that $X^1=V_1$.  Attach $g$ $2$-handles whose core disks form a complete collection of disks for the handlebody $H_\gamma$, so that $X^2$ is the union of $X^1$ and regular neighborhood of these disks.  The remainder $\overline{X\setminus X^2}$ is diffeomorphic to $V_2\natural V_3$.  Choose $(k_2+k_3)$ $3$-handles, $k_2$ of which lie in $V_2$ and $k_3$ of which lie in $V_3$, so that $\overline{X\setminus X^3}$ is a 4-ball that forms the single $4$-handle of the decomposition.

Recall that the homology of any smooth manifold $X$ with a handle decomposition can be calculated from the handle complex
\[
	\cdots\xrightarrow{\hat{\partial}_{i+1}} \mathcal{C}_i
    	\xrightarrow{\hat{\partial}_i} \mathcal{C}_{i-1}
        \xrightarrow{\hat{\partial}_{i-1}} \mathcal{C}_{i-2}
        \xrightarrow{\hat{\partial}_{i-2}}\cdots,
\]
where $\mathcal{C}_i$ is the free abelian group generated by the $i$-handles of the decomposition.  The
boundary map $\hat{\partial}_i$ is defined on each $i$-handle $h$ by
\begin{equation}\label{eq:HandleBoundaryMap}
	\hat{\partial}_i(h)=\sum_j\tau_jh_j,
\end{equation}
where the sum is taken over all $(i-1)$-handles $h_j$, and $\tau_j$ is the algebraic intersection number of the attaching sphere of $h$ with the belt sphere of $h_j$ in $\partial X^{i-1}$ \cite{Kosinski}.
Because our manifold is orientable and there is only one $0$-handle and one $4$-handle,
	we have $\mathcal{C}_0=\mathcal{C}_4=\Z$ and $\partial_1=0$, $\partial_4 = 0$. Thus we will focus on the nontrivial part of the sequence:
\begin{equation}
	\label{eq:HandleComplex}
	0 \rightarrow \mathcal{C}_3
    \xrightarrow{\hat{\partial}_3} \mathcal{C}_2
    \xrightarrow{\hat{\partial}_2} \mathcal{C}_1
    \rightarrow 0.
\end{equation}

Let $b_1:\mathcal{C}_1\rightarrow H_2(\partial V_1)$ be the isomorphism which sends each $1$-handle generator of $\mathcal{C}_1$ to the element of $H_2(\partial V_1)$ represented by its belt sphere, and let $a_2:\mathcal{C}_2\rightarrow L_\gamma$ be the isomorphism which sends each $2$-handle generator of $\mathcal{C}_2$ to the element of $L_\gamma\subset H_1(\Sigma)$ that is represented by its attaching sphere (which by construction is embedded in $\Sigma$).  In this notation, (\ref{eq:HandleBoundaryMap}) takes the following form:
  \begin{equation}
      \hat{\partial}_2(h)=\sum_j \langle b_1(h_j),\iota_{\partial V_1} a_2(h)\rangle_{\partial V_1} h_j.
  \end{equation}
The map $\iota_{\partial V_1}:H_1(\Sigma)\rightarrow H_1(\partial V_1)$ is induced by inclusion.
By Lemma \ref{lem:3dHomology}, this equation can in turn be replaced by
  \begin{equation}
  	\hat{\partial}_2(h)=\sum_j\langle \partial b_1(h_j),a_2(h)\rangle_\Sigma \, h_j.
  \end{equation}
Thus if $\partial_*$ is the Mayer-Vietoris map of Lemma \ref{lem:3dHomology}, $D:L_\alpha\cap L_\beta\rightarrow \textnormal{Hom}(L_\alpha\cap L_\beta,\Z)$ is the dual map induced by the basis elements $\partial_*b_1(h_j)$, and we define $f_1=D\circ\partial \circ b_1$ and $f_2=a_2$, then we obtain
  \begin{equation}\label{eq:RightHalf}
  	\partial_2\circ f_2=f_1\circ \hat{\partial}_2.
  \end{equation}

By construction, $\mathcal{C}_3$ splits as $\mathcal{C}_3^2\oplus\mathcal{C}_3^3$, where $\mathcal{C}_3^2$ is the free abelian group generated by the $3$-handles defining $V_2$ and $\mathcal{C}_3^3$ is its analogue for $V_3.$  Thus we obtain an isomorphism $a_3:\mathcal{C}_3\rightarrow H_2(\partial V_2)\oplus H_2(\partial V_3)$ which maps each $3$-handle generator of $\mathcal{C}_3^2$ to the element of $H_2(\partial V_2)$ represented by its attaching sphere in $\partial V_2$, and likewise for the generators of $\mathcal{C}_3^3$ and $H_2(\partial V_3)$.  Let $f_3:\mathcal{C}_3\rightarrow (L_\alpha\cap L_\gamma)\oplus (L_\beta\cap L_\gamma)$ be the map obtained by composing $a_3$ with the sum of the $\partial_*$ maps $H_2(\partial V_2)\rightarrow L_\alpha\cap L_\gamma$ and $H_2(\partial V_3)\rightarrow L_\beta\cap L_\gamma$ from Lemma \ref{lem:3dHomology}.  We now prove that
\begin{equation}\label{eq:LeftHalf}
\partial_3\circ f_3 = f_2\circ \hat{\partial}_3.
\end{equation}
This, together with \eqref{eq:RightHalf}, will show chain complex in \eqref{eq:HandleComplex} is isomorphic to (the middle part of) the complex in the statement of the theorem, which will complete the proof.

Recall that we chose $2$-handles for the handle decomposition of $X$ whose cores form a defining collection of disks for $W_\gamma$.  Thus every sphere $S\subset \partial V_2$ (or $\partial V_3$) is isotopic to a sphere which intersects $W_\gamma$ only in disks that are parallel to the core disks of the $2$-handles. (This follows from a standard innermost and outermost disk argument.) If a component $E$ of $S\cap W_\gamma$ is isotopic in $W_\gamma$ to the core of the $2$-handle $h_j$, then it is disjoint from the belt sphere of all other $2$-handles, and meets the belt sphere of $h_j$ in a single intersection point.  Moreover, this intersection is positive if the orientation of $E$ matches that of the core of $h_j$, and negative otherwise.  Thus if $h$ is a $3$-handle of the decomposition, then in \eqref{eq:HandleBoundaryMap} above $\tau_j$ counts (with sign) the number of parallel copies of the core of $h_j$ that lie in $S\cap W_\gamma$, where $S$ is the attaching sphere of $h$.  It follows that, for every generator $h\in \mathcal{C}_3$, $f_2\circ\partial_3(h)$ is the element of $H_1(\Sigma)$ represented by $S\cap \Sigma$.
The same description applies to $\partial_3\circ f_3(h)$, so we have shown that \eqref{eq:LeftHalf} holds.
\end{proof}

The chain complex of Theorem \ref{thm:4dHomology} allows us to find presentation matrices for $H_1(X)$, but aside from this it does not appear to yield any particularly elegant characterizations of $H_1(X)$.
However, we do obtain the following descriptions of $H_2(X)$ and $H_3(X)$ which are symmetric with respect to the $\alpha$, $\beta$, and $\gamma$ curves.

\begin{corollary}\label{cor:H2Formula}
Writing $\{a+b+c=0\}$ for $\{(a,b,c)\in L_\alpha\times L_\beta\times L_\gamma|a+b+c=0\},$ we have
\begin{equation}\label{eq:symH2Formula}H_2(X)\cong \frac{\{a+b+c=0\}}{\{ c=0\}+\{b=0\}+\{a=0\}},\end{equation} where the summands in the denominators are understood to be the subgroups of $\{a+b+c=0\}$ satisfying the given linear equations.
\end{corollary}

\begin{proof}
First we prove that
\begin{equation}\label{eq:H2FormulaNonSymmetric}
H_2(X)\cong \frac{L_\gamma\cap(L_\alpha+L_\beta)}{(L_\gamma\cap L_\alpha)+(L_\gamma\cap L_\beta)}.
\end{equation}
If $\partial_2$ and $\partial_3$ are the maps from the complex in Theorem \ref{thm:4dHomology}, then $\ker({\partial}_2)$ consists of those elements of $L_\gamma$ which are orthogonal to $L_\alpha\cap L_\gamma$ in $H_1(\Sigma)$.  By Lemma \ref{lem:Lagrangian}, $L_\alpha^{\perp}=L_\alpha$ and $L_\beta^{\perp}=L_\beta,$ so $(L_\alpha\cap L_\beta)^\perp=L_\alpha+L_\beta$.  Hence $\ker({\partial}_2)=L_\gamma\cap (L_\alpha+L_\beta).$  Additionally, one can readily see that
$\textnormal{Im}(\partial_3)=(L_\gamma\cap L_\alpha)+(L_\gamma\cap L_\beta)$, so
\eqref{eq:H2FormulaNonSymmetric} now follows from Theorem \ref{thm:4dHomology}.

To prove the symmetric formula holds, first note that $(a,b,c)\mapsto c$ induces
\[
	\frac{\{(a,b,c)\in L_\alpha\times L_\beta\times L_\gamma|a+b-c=0\}}{\{a+b=0\}}
    	\cong L_\gamma\cap(L_\alpha+L_\beta),
\]
and so
\[
	\frac{L_\gamma\cap(L_\alpha+L_\beta)}{(L_\gamma\cap L_\alpha)+(L_\gamma\cap L_\beta)}
    	\cong \frac{\{(a,b,c)\in L_\alpha\times L_\beta\times L_\gamma|a+b-c=0\}}
        			{\{a+b=0\}+\{a-c=0\}+\{b-c=0\}}.
\]
The isomorphism $H_1(\Sigma)^3\to H_1(\Sigma)^3$ induced by $(a,b,c)\mapsto(a,b,-c)$ induces an
isomorphism with
\[
	\frac{\{(a,b,c)\in L_\alpha\times L_\beta\times L_\gamma|a+b+c=0\}}{\{a+b=0\}+\{a+c=0\}+\{b+c=0\}}
\]
which is equal to the desired expression.
\end{proof}

\begin{corollary}\label{cor:H3Formula}
	$H_3(X)\cong L_\alpha\cap L_\beta\cap L_\gamma.$
\end{corollary}

\begin{proof}
From Theorem \ref{thm:4dHomology} we have $H_3(X)\cong \ker({\partial}_3)$.
	Suppose that $(x,y) \in (L_\alpha \cap L_\gamma) \oplus (L_\beta \cap L_\gamma)$
    and ${\partial}_3(x,y)=0$. Then $x=-y$ and, since $y\in L_\beta\cap L_\gamma$, we have $x\in L_\beta$.  Thus the map $f:L_\alpha\cap L_\beta\cap L_\gamma\rightarrow \ker({\partial}_3)$ given by $f(x)=(x,-x)$ is an isomorphism.
\end{proof}

\begin{definition}\label{def:FormDefinition}
Let $\Phi:(L_\gamma\cap(L_\alpha+L_\beta))\times(L_\gamma\cap(L_\alpha+L_\beta))\rightarrow \Z$ be given by the formula
\begin{equation}
\Phi(x,y)=\langle x',y\rangle_\Sigma,
\end{equation}
where $x'$ is any element of $L_\alpha$ which satisfies $x-x'\in L_\beta$.
\end{definition}

The element $x'$ above exists since $x\in L_\alpha+L_\beta$, and although it will not generally be unique, the value of $\Phi(x,y)$ does not depend on the choice of $x'$. The proof of this is formally identical to the last paragraph of the proof of Lemma \ref{lem:LinkingNumber}.

\begin{theorem}\label{thm:IntersectionForm}
The group
\[
	\frac{L_\gamma\cap(L_\alpha+L_\beta)}{(L_\gamma\cap L_\alpha)+(L_\gamma\cap L_\beta)}
\]
together with the form induced by $\Phi$ 
is isomorphic to $H_2(X)$ together with the intersection form.
\end{theorem}

\begin{proof}
Define a handle decomposition of $X$ as in the proof of Theorem \ref{thm:4dHomology}, so that $X^1=V_1$ and $X^2$ is the union of $V_1$ with a regular neighborhood of a defining collection of disks for $H_\gamma$;
we retain the notation for $\mathcal{C}_i$ and $\partial_i$ as well.

There is an isomorphism $f:H_2(X^2)\rightarrow L_\gamma\cap(L_\alpha+L_\beta)$ which factors through $\ker(\partial_2)$.  Explicitly, given any $z\in H_2(X^2)$ and an oriented surface $Q$ representing it, let
  \begin{equation}
  \tilde{f}(z)=\sum_{j=1}^g \tau_j h_j,
  \end{equation}
where $h_1,\ldots , h_g$ is the collection of $2$-handles generating $\mathcal{C}_2$ and $\tau_j$ is the algebraic intersection number of $Q$ and the belt sphere of $h_j\in \mathcal{C}_2$.
Note that, under the usual identification of $\mathcal{C}_2$ with $H_2(X^2,X^1)$, $\tilde{f}$ corresponds to the map $H_2(X^2)\rightarrow H_2(X^2,X^1)$ coming from the long exact sequence of the pair $(X_2,X_1)$.
Thus we see that $\tilde{f}$ is injective injective with image $\ker(\partial_2)$.  Composing $\tilde{f}$ with the isomorphism $\ker(\partial_2)\rightarrow L_\gamma\cap (L_\alpha\cap L_\beta)$ which sends each $h_j$ to its attaching sphere in $\Sigma$ yields the isomorphism $f$ we are interested in.

We claim that
  \begin{equation}\label{eq:FormEquality}
      \langle w,z\rangle_{X^2}=\Phi(f(w),f(z))
  \end{equation}
for all $w,z\in H_2(X^2).$
Let $\{\gamma_1,\ldots ,\gamma_g\}$ be a defining set of oriented curves for $H_\gamma$.  Given $w,z\in H_2(X^2)$, let $J$ and $K$ be disjoint unions of embedded curves from $\{\gamma_1,\ldots ,\gamma_g\}$ such that $f(w)=[J]$ and $f(z)=[K]$. (It may be necessary for $J$ and $K$ to each to contain multiple parallel copies of the same oriented curve $\gamma_i$.)  Since $J$ and $K$ are null-homologous in $\partial X_1=H_\alpha\cup H_\beta$, there exist Seifert surfaces $\tilde{F}$ and $\tilde{G}$ for $J$ and $K$ embedded in $\partial X_1.$  If $F$ and $G$ the surfaces obtained by capping off $\tilde{F}$ and $\tilde{G}$ with disks in $H_\gamma$, then $[F]=w$ and $[G]=z$.  We may assume that $F\cap H_\gamma$ and $G\cap H_\gamma$ are disjoint from one another, and by pushing the interior of $\tilde{G}$ into the interior of $X^1$ so that $G\cap \partial X^1=K$, we deduce that
\[
	\langle w,z\rangle_{X^2}
    	=\langle [F],[G]\rangle_{X^2}
        =\langle [\tilde{F}],[K]\rangle_{\partial X^1}
        =\textnormal{lk}(J,K).
\]
Lemma \ref{lem:LinkingNumber} implies that $\textnormal{lk}(J,K)=\Phi([J],[K])$.  Hence we have proven (\ref{eq:FormEquality}), and can conclude that $\Phi$ is the intersection form on $H_2(X^2)$.  The theorem now follows, since the map $H_2(X^2)\rightarrow H_2(X)$ induced by inclusion is a surjection which preserves intersection numbers, and the kernel of this map corresponds to $(L_\gamma\cap L_\alpha)+(L_\gamma\cap L_\beta)$ inside of $L_\gamma\cap(L_\alpha+L_\beta).$
\end{proof}

\begin{rmk}
The intersection from in the previous theorem is precisely the form $-\sigma(V;A,B,C)$ appearing in \cite{Wall} with $V=H_1(\Sigma)$, $A=L_\alpha$, $B=L_\beta$, and $C=L_\gamma.$  This may be unsurprising, since the main theorem of that paper implies that the intersection form of $X$ and $-\sigma(H_1(\Sigma);L_\alpha,L_\beta,L_\gamma)$ must have the same signature.  However, the fact that the two forms are identical is a special property of trisections and does not follow from \cite{Wall}.
\end{rmk}

\section{Applications}
\label{sec:applications}

With Theorems \ref{thm:4dHomology} and \ref{thm:IntersectionForm} in hand, we are ready to give some applications.  The foundation is a description of how to compute the second homology and intersection form of a $4$-manifold from an arbitrary trisection diagram.  The only data we need from the diagram is the triple of algebraic intersection matrices that is associated with it.

\begin{definition}
	\label{def:IntersectionMatrix}
Let $(\Sigma;\alpha,\beta,\gamma)$ be a trisection diagram of $X$, where for $\epsilon \in \{\alpha,\beta,\gamma\}$
we have $\epsilon=\epsilon_1\cup\cdots \cup \epsilon_g$ and each curve has been assigned an orientation.  Let $_\alpha Q_\beta$ be the matrix of intersection numbers $\langle [\alpha_i],[\beta_j]\rangle_\Sigma,$ and likewise for $_\beta Q_\alpha$, $_\gamma Q_\beta$, etc.
\end{definition}

The oriented trisection diagram $(\Sigma;\alpha,\beta,\gamma)$ determines bases of $L_\alpha$, $L_\beta,$ and $L_\gamma$, and the intersection matrices defined above represent the action of $\langle\cdot ,\cdot\rangle_\Sigma$ on $L_\alpha$, $L_\beta$ and $L_\gamma$ with respect to these bases.  For example, by Lemma \ref{lem:Lagrangian}, the subspace $L_\alpha\cap L_\beta$ can be expressed in terms of the $\alpha$-curves as the kernel of $_\beta Q_\alpha$ or in terms of the $\beta$-curves as the kernel of $_\alpha Q_\beta.$  By Corollary \ref{cor:H3Formula}, it follows that $H_3(X)$ is described in terms of the $\gamma$-curves by the formula
  \begin{equation}
      \label{eqn:easyH3}
      H_3(X)\cong \ker(_\alpha Q_\gamma)\cap\ker(_\beta Q_\gamma).
  \end{equation}
Using similar reasoning and Corollary \ref{cor:H2Formula}, we can also describe $H_2(X)$ in terms of the $\gamma$-curves by
  \begin{equation}
      \label{eqn:easyH2}
  H_2(X)\cong\frac{\ker(B\cdot {_\alpha}Q_\gamma)}{\ker(_\alpha Q_\gamma)+\ker(_\beta Q_\gamma)},
  \end{equation}
where $B$ is any matrix of row vectors which form a basis of
$\ker(_\beta Q_\alpha)=L_\alpha\cap L_\beta$.

In the following proposition, $I_g^k$ denotes the $g\times g$ matrix which is the identity matrix in its upper left $k\times k$ minor and zero elsewhere.  By Waldhausen's theorem on uniqueness of Heegaard splitting surfaces for $S^3$ \cite{Waldhausen}, we can always perform handleslides on the $\alpha$- and $\beta$-curves so that $_\alpha Q_\beta =I_g^{g-k_1}$.

\begin{proposition}\label{prop:GeneralMatrix}
Suppose $(\Sigma;\alpha,\beta,\gamma)$ is a $(g;k_1,k_2,k_3)$-trisection diagram, chosen so that $_\alpha Q_\beta=I_g^{g-k_1}.$  If $A$ is a matrix of column vectors in $L_\gamma\cap (L_\alpha+L_\beta)$ (expressed with respect to the $\gamma$-basis) which represents a basis of $H_2(X)/\textnormal{torsion}$, then $\Phi$ is represented by the matrix
\[
	A^T{_\gamma}Q_\beta I_g^{g-k_1} {_\alpha}Q_\gamma A
\]
with respect to this basis.
\end{proposition}

\begin{proof}
Recall from Definition \ref{def:FormDefinition} that, for a pair of elements $x,y\in L_\gamma\cap(L_\alpha+L_\beta)$, $\Phi(x,y)=\langle x',y\rangle_\Sigma$ for some $x'\in L_\alpha$ such that $x-x'\in L_\beta$. Assuming $x$ has been expressed in terms of the $\gamma$-basis, this means we must find a solution $x'$ to the equation
  \begin{equation}\label{eq:Setup}
	_\beta Q_\alpha x'={_\beta}Q_\gamma x.
  \end{equation}
Such an $x'$ will be expressed with respect to the $\alpha$-basis.  Since we have assumed that $_\alpha Q_\beta=I_g^{g-k_1}$, it follows that $_\beta Q_\alpha=-I_g^{g-k_1}$, and so we can assume that $x'$ has been chosen so that its lower $g-k_1$ entries are all $0$. Thus we have
\[
	x'=-I_g^{g-k_1}{_\beta}Q_\gamma x.
\]
Since $x'$ is expressed with respect to the $\alpha$-basis and $y$ is to be expressed with respect to the $\gamma$-basis, the product $\langle x',y\rangle_\Sigma$ is given by
\[
	(x')^T{_\alpha}Q_\gamma y
    	= (-I_g^{g-k_1}{_\beta}Q_\gamma x)^T{_\alpha}Q_\gamma y
        = x^T{_\gamma}Q_\beta I_g^{g-k_1}{_\alpha}Q_\gamma y. \qedhere
\]
\end{proof}

In the case of a $(g;0,k_2,k_3)$-trisection, we can relax the triviality assumption on $_\alpha Q_\beta$ and still obtain a nice representation of $\Phi.$  Note that all 4-manifolds which admit handle decompositions without 1-handles will admit such trisections by the above construction.

\begin{proposition}\label{prop:LovelyProduct}
Suppose $(\alpha,\beta,\gamma)$ is a $(g;0,k_2,k_3)$-trisection diagram.  Then $\Phi$ is given on $L_\gamma$ by the matrix $$\Phi = {_\gamma}Q_\beta({_\alpha}Q_\beta)^{-1}{_\alpha}Q_\gamma.$$
\end{proposition}

\begin{proof}
The proof is a computation that follows the pattern of Proposition \ref{prop:GeneralMatrix}, but
in this case $H_\alpha$ and $H_\beta$ form a Heegaard splitting of $S^3$. Thus
$L_\gamma\cap (L_\alpha+L_\beta)=L_\gamma$ and $_\beta Q_\alpha$ is invertible, so we can
solve \eqref{eq:Setup} to obtain $x'=({_\beta}Q_\alpha)^{-1}{_\beta}Q_\gamma x$. For any $x,y\in L_\gamma$, we therefore have
\[
	\Phi(x,y)
    	=(({_\beta}Q_\alpha)^{-1}{_\beta}Q_\gamma x)^T{_\alpha}Q_\gamma y
        =x^T {_\gamma}Q_\beta({_\alpha}Q_\beta)^{-1}{_\alpha}Q_\gamma y.
\]
This completes the proof.
\end{proof}

By \cite{MeierSchirmerZupan}, every $4$-manifold which admits a handle decomposition without $1$- or $3$-handles admits a $(g;0,k_2,0)$-trisection, and we can make the following statement about such trisections.

\begin{theorem}\label{thm:AlgebraicTriviality}
Suppose $\mathcal{T}$ is a $(g;0,k_2,0)$-trisection of $X$, and let $Q_X$ be any matrix representing the intersection form of $X$.  Then $\mathcal{T}$ admits a diagram $(\alpha,\beta,\gamma)$ with the following properties:
\begin{enumerate}
  \item $(\alpha,\beta)$ is a standard diagram of $S^3$, with indices labeled so that $|\alpha_i\cap \beta_j|=\delta_{ij}$
  \item $_\gamma Q_\beta =\textnormal{Id}$
  \item   The matrix $_\alpha Q_\gamma $ has $Q_X$ as its upper left minor, and is zero elsewhere.
\end{enumerate}
\end{theorem}

\begin{proof}
Given any diagram $(\alpha,\beta,\gamma)$ of $\mathcal{T}$, by Waldhausen's theorem on the uniqueness of Heegaard splitting surfaces for $S^3$ \cite{Waldhausen} we can perform handle slides of the $\alpha$ and $\beta$ curves so that Condition (1) is satisfied.  Moreover we can orient the components of $\alpha$ and $\beta$ so that $_\alpha Q_\beta =({_\alpha}Q_\beta )^{-1}=\textnormal{Id}.$  Since $_\gamma Q_\beta$ is a presentation matrix of $H_1(S^3)=0$ it is unimodular, so after a sequence of row operations it can be reduced to the identity matrix, and these row operations can be geometrically realized.

Specifically, the act of multiplying the $i$th row of $_\gamma Q_\beta$ by $-1$ corresponds to reversing the orientation of $\gamma_i$, the act of switching row $i$ with row $j$ corresponds to swapping the indices of $\gamma_i$ and $\gamma_j$, and the act of adding or subtracting row $i$ from row $j$ corresponds to performing a handle slide of $\gamma_i$ over $\gamma_j$ and then replacing $\gamma_j$ with the resulting curve.  Hence, after a sequence of handle slides, orientation changes, and re-indexing moves on the $\gamma$ curves, Conditions (1) and (2) will both be satisfied by $(\alpha,\beta,\gamma)$.

This being done, by Proposition \ref{prop:LovelyProduct} the matrix $_\alpha Q_\gamma$ can be converted to the form required by Condition (3) by a sequence of double row/column swaps, double multiplication of rows and columns by $-1$, and moves $\pm d_{ij}$, where $d_{ij}$ denotes the operation on $_\alpha Q_\gamma$ which adds the $i$th column to the $j$th column, and then adds the $i$th row to the $j$th row, and $-d_{ij}$ denotes the move which subtracts rows and columns from each other.  We leave it to the reader to verify that the first two moves can be realized by index changes and orientation changes, respectively, and in the following paragraph we explain how the $d_{ij}$ moves can be realized geometrically as well, in a way that maintains Conditions (1) and (2).

Specifically, the column move in $d_{ij}$ can be geometrically realized by replacing $\alpha_j$ with a curve obtained by sliding $\alpha_i$ over $\alpha_j$, and similarly the row move corresponds to replacing $\gamma_j$ with a curve obtained by sliding $\gamma_i$ over $\gamma_j$.  These handleslides will change $_\alpha Q_\beta$ by adding its $i$th row to its $j$th row, and they will change $_\gamma Q_\beta$ by adding its $i$th column to its $j$th column.  However, so long as we have chosen the arc defining our handleslide of $\alpha_i$ over $\alpha_j$ to be disjoint from all of the $\alpha$ and $\beta$ curves except at its endpoints (as can always be done), then we can alway perform a cancellation move which replaces the curve $\beta_i$ with the result of sliding of $\beta_i$ over $-\beta_j$ in the manner shown in Figure \ref{fig:DoubleHandleslide}.  This geometric move makes the $\alpha$ and $\beta$ curves satisfy Condition (1) again, and on the algebraic level it corresponds to a subtraction of the $j$the column of $_\alpha Q_\beta$ from its $i$th column, and a subtraction of the $j$th row of $_\gamma Q_\beta$ from its $i$th row.  Thus we can geometrically realize the move $d_{ij}$ while maintaining Conditions (1) and (2), and the same reasoning works for $-d_{ij}.$
\end{proof}

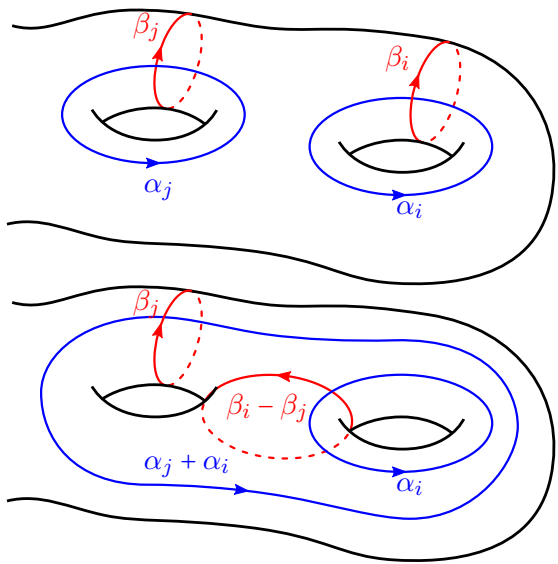
\begin{figure}
	\begin{center}
    \begin{overpic}[width=0.5\textwidth]{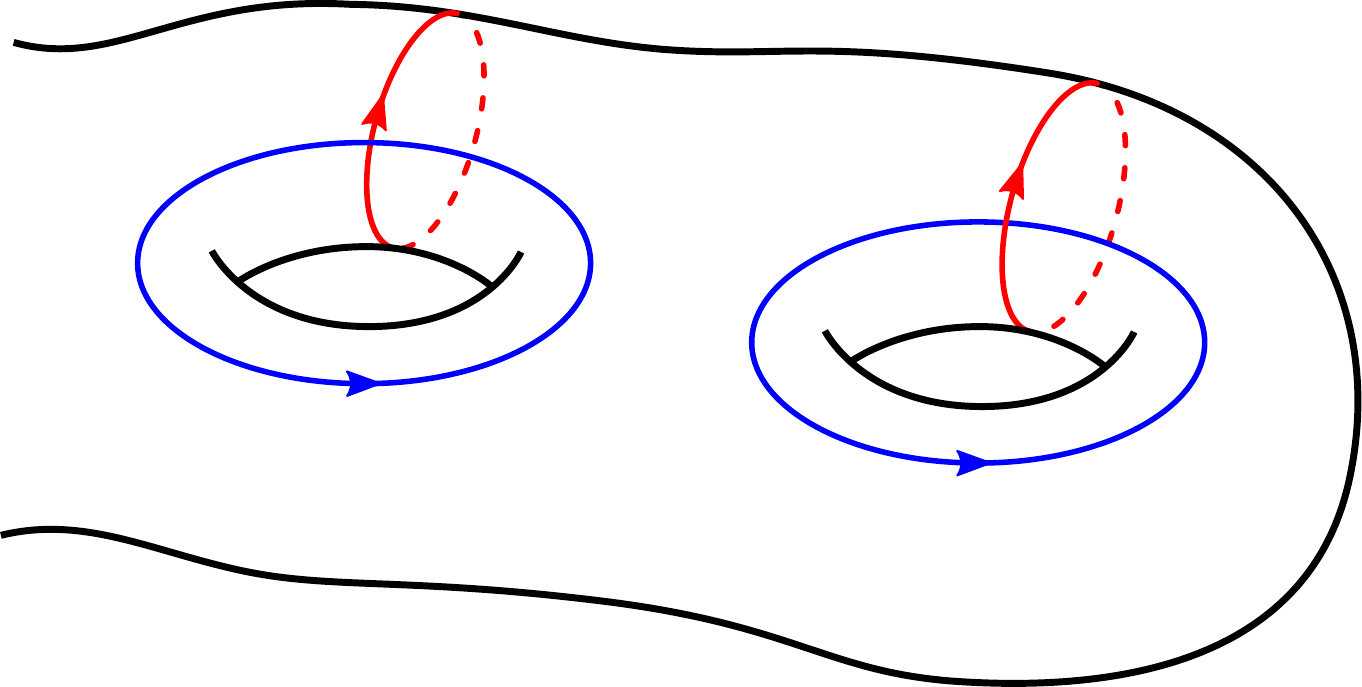}
        \put (23,46) {\color{red}$\beta_j$}
        \put (71,13) {\color{blue}$\alpha_i$}
        \put (69,40) {\color{red}$\beta_i$}
        \put (25,17) {\color{blue}$\alpha_j$}
    \end{overpic}

    \begin{overpic}[width=0.5\textwidth]{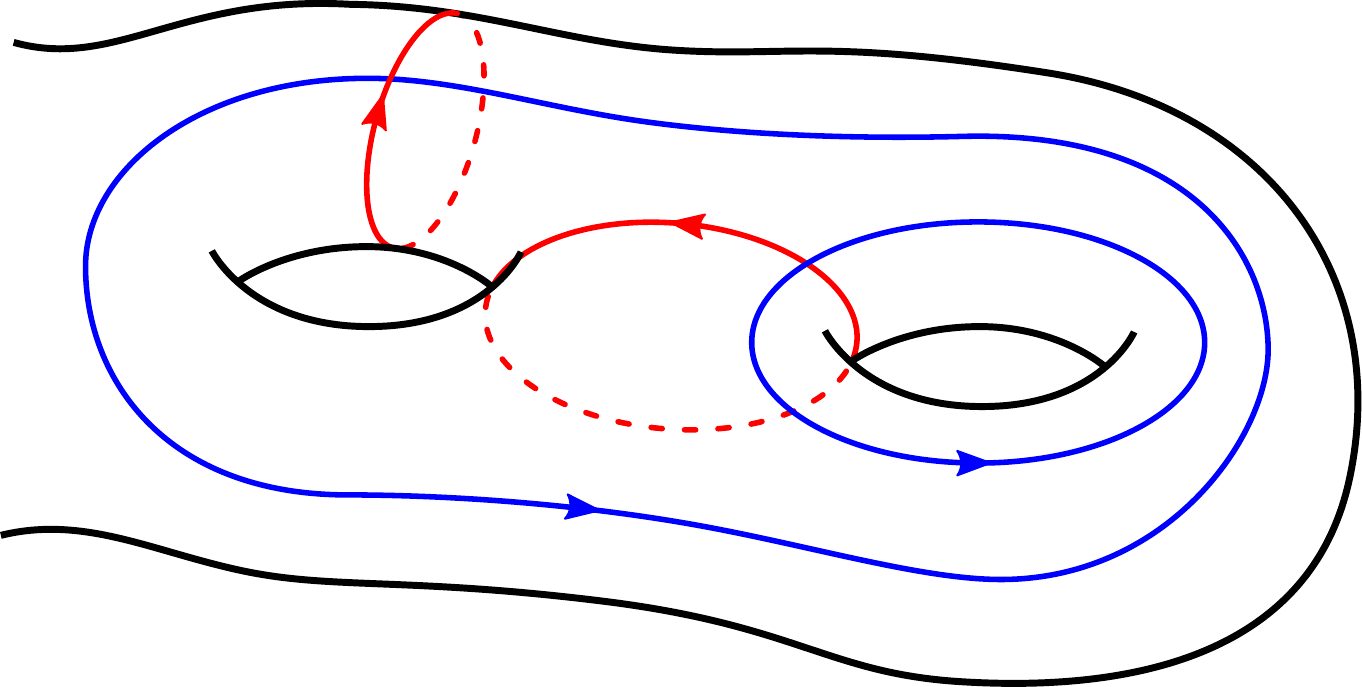}
        \put (23,46) {\color{red}$\beta_j$}
        \put (71,13) {\color{blue}$\alpha_i$}
        \put (40,27) {\color{red}$\beta_i-\beta_j$}
        \put (25,17) {\color{blue}$\alpha_j+\alpha_i$}

    \end{overpic}

    \end{center}
    \caption{A double handleslide which maintains duality.}
    \label{fig:DoubleHandleslide}
\end{figure}


It is tempting to hope that Theorem \ref{thm:AlgebraicTriviality} can be used to prove the triviality of a large class of trisections.  Specifically, if one could show that it is always possible to find a trisection diagram satisfying the equation $|\omega \cap \sigma|=|\langle \omega,\sigma\rangle_\Sigma |$ for all pairs of curves $\omega$ and $\sigma$ appearing in the diagram, then the triviality or near-triviality (in the case when the form contains $E_8$ summands) of all $(g;0,k_2,0)$-trisections would follow by choosing an appropriate $Q_X$ in Theorem \ref{thm:AlgebraicTriviality}.  However, a paper of Meier and Zupan appearing in this collection~\cite{MeierZupan} gives examples of exotic connected sums of $\mathbb{CP}^2$'s and $\overline{\mathbb{CP}^2}$'s admitting $(g;0)$ trisections.  Such manifolds cannot admit geometrically trivial trisection diagrams, even though Theorem \ref{thm:AlgebraicTriviality} tells us that such manifolds will admit homologically trivial diagrams.

Moreover, even when a trisection admits a trivial diagram, other diagrams of the same trisection can be chosen in which the discrepancy between geometric intersection numbers and algebraic intersections numbers is arbitrarily large.  For example, the standard $(2;0)$ trisection of $\mathbb{CP}^2\# \mathbb{CP}^2$ does indeed admit a trivial diagram $(\alpha,\beta,\gamma)$ such that $|\alpha_i\cap \beta_j|=|\alpha_i\cap \gamma_j|=|\beta_i\cap \gamma_j|=\delta_{ij}$ for $1\leq i,j,\leq 2.$  However, there exist separating curves $\sigma_1$ and $\sigma_2$ on
$\Sigma$ which bound disks in both $H_\alpha$ and $H_\beta$, and such that $\Sigma\setminus (\sigma_1\cup\sigma_2)$ consists entirely of disks.  Thus if $f:\Sigma\rightarrow \Sigma$ is the composition of the positive Dehn twist about $\sigma_1$ followed by the negative Dehn twist about $\sigma_2$, then $f$ is a pseudo-Anosov map  (by Penner's construction \cite{Penner}) lying in the intersection of the mapping class group of $(\Sigma; W_\alpha,W_\beta)$ and the Torelli subgroup of $\Sigma$ .  If $d_\Sigma$ is simpicial distance in the curve complex, it follows that $d_\Sigma(\gamma_i,f^n(\gamma_i))$ goes to infinity as $n$ does, which in turn implies that the quantities $|f^n(\gamma_i)\cap \alpha_i|$ go to infinity as well (as was shown by Hempel \cite{Hempel}).  Nevertheless, the diagrams $(\alpha,\beta,f^n(\gamma))$ satisfy the conditions of Theorem \ref{thm:AlgebraicTriviality}, and in fact define the standard genus $2$ trisection of $\mathbb{CP}^2\#\mathbb{CP}^2.$

\bibliographystyle{plain}

\end{document}